\documentclass{amsart}
\usepackage[latin1]{inputenc}
\usepackage{amsmath}
\usepackage{amsthm}
\usepackage{amsfonts}
\usepackage{amssymb}
\usepackage{enumerate}
\usepackage{mathrsfs}
\usepackage[all]{xy}
\usepackage{epic}
\usepackage{graphicx}
\usepackage[hypertex]{hyperref}
\usepackage{placeins}

\newcommand\bk{{\mathbb K}}
\newcommand\bp{{\mathbb P}}
\newcommand\bfff{{\mathbb F}}

\newcommand\bc{{\mathbb C}}
\newcommand\bq{{\mathbb Q}}

\newcommand\bz{{\mathbb Z}}

\newtheorem{thm}{Theorem}
\newtheorem{conj}{Conjecture}
\newtheorem{prop}[thm]{Proposición}
\newtheorem{cor}[thm]{Corolary}
\newtheorem{lema}[thm]{Lemma}

\theoremstyle{remark}

\newtheorem{obs}[thm]{Remark}
\theoremstyle{definition}
\newtheorem{dfn}[thm]{Definition}
\newtheorem{ntc}[thm]{Notation}
\newtheorem{ejm}[thm]{Example}

\newtheorem{alg}{Algorithm}
\numberwithin{equation}{section}
\numberwithin{figure}{section}

\title{A polynomial generalization of the Euler characteristic for algebraic sets.}
\author{Miguel A. Marco-Buzun\'ariz}
\address{Departamento de Matem\'aticas, Universidad de Zaragoza,
C/ Pedro Cerbuna 12, 50009 Zaragoza, Spain.}
\email{mmarco@unizar.es}
\thanks{Partially supported by the ERC Starting Grant project TGASS, by Spanish Contract MTM2010-21740-C02-01, and by research group ``Geometría'' E15 of the government of Aragón.}

\begin{document}
 \begin{abstract}
 We present a method to compute the Euler characteristic of an algebraic subset of $\bc^n$. This method relies on clasical tools such as Gröbner basis and primary decomposition. The existence of this method allows us to define a new invariant for such varieties. This invariant is related to the poblem of counting rational points over finite fields.
\end{abstract}

\maketitle
\section{Introduction}
One of the main invariants of a topological space is its Euler characteristic. It was initially defined for cell complexes, but several extensions have been defined to more general classes of spaces. In the setting of complex algebraic varieties, the natural extension is the Euler characteristic with compact support. In~\cite{szafraniec-complex} Szafraniec gives a method to compute the Euler characteristic of a complex algebraic set by using methods from the real geometry. In this paper, we present another method, that only makes use of the basic properties of the Euler characteristic, and classical results on algebraic sets. This way of computing the Euler characteristic gives naturally a stronger invariant, which we define.

The method works as follows. Consider $V\subseteq \bc^n$ an irreducible algebraic set of dimension $d$ and degree $g$. Take a generic linear projection $\pi:\bc^n\to \bc^d$. If we consider $\pi$ restricted to $V$, it is a $g:1$ branched cover. The branchig locus $\Delta$ and its preimage $\pi\mid_V^{-1}(\Delta)$ can be computed. From the aditivity and the multiplicativity for covers of the Euler characeristic, we have the following formula:
\[
 \chi(V)=g\cdot \chi(\bc^d)-g\cdot \chi(\Delta)+\chi(\pi\mid_V^{-1}(\Delta)).
\]
So the computation of $\chi(V)$ is reduced to the computation of the Euler characteristic of algebraic sets of lower dimension, allowing us to use a recursion process.

In the previous method, we make use of the fact that $\chi(\bc^d)=1$. If instead of making this substitution, we keep track of $\chi(\bc^d)$ as a formal symbol, we obtain a stronger invariant $F(V)$. This invariant is defined as a polynomial in $\bz[L]$, and has some interesting properties: the dimension, degree, and Euler characteristic of an algebraic set can be computed from this polynomial. It also gives information on the number of points on some varieties over finite fields. This relation with finite varieties could be used to compute this invariant by counting points.

In Sections~\ref{sec-justificacion} and \ref{sec-descripcion} we show the preliminary results that prove the correctness of the method to compute the Euler characteristic, and describe the algorithm. Section~\ref{sec-invariante} is devoted to the generalization of this method to the new invariant, which is also defined and some of its properties are shown. The extension of this invariant to projective varieties 
is discused in Section~\ref{sec-proyectivo}. In Sections~\ref{sec-ejemplos} and \ref{sec-codigo} we include an implementation of the two algorithms in Sage, together with some examples and timings. As an important example, we show that in the case of hyperplane arrangements this invariant coincides with the characteristic polynomial. Finally, the relationship of the invariant with the number of points over finite fields is shown in Section~\ref{sec-cuerposfinitos}.

\section{Theoretical justification}
\label{sec-justificacion}
Let $V=V(I)\subseteq \bc^n$ be the algebraic set determined by a radical ideal $I$. Without loss of generality, we can assume that it is in general position (in the sense that we will precise later). By computing the associated primes of $I$ we obtain the decomposition in irreducible components $V=V_1 \cup \cdots \cup V_c$. The Euler characteristic $\chi(V)$ can be expressed as $\chi(V_1)-\chi((V_1)\cap(V_2\cup\cdots\cup V_c))+\chi(V_2\cup\cdots\cup V_c)$. The variety $(V_1)\cap(V_2\cup\cdots\cup V_c)$ is an algebraic set of lower dimension. So, by a double induction argument (over the dimension and over the number of irreducible components), we may reduce the problem of computing $\chi(V)$ to the case where $V$ is either zero-dimensional or irreducible. 

If $V$ is zero-dimensional, it consists on a number of isolated points, and its Euler characteristic equals the number of points. This number of points can be computed as the degree of the homogenization of the radical of $I$ (which can be computed via the Hilbert polynomial, see~\cite[Chapter 5]{singular-book} for example).

For the case of an irreducible variety $V=V(I)\subseteq\bc^n$ being $I\trianglelefteq\bc[x_1,\ldots,x_n]$ a radical ideal of Krull dimension $d$, we will distinguish the homogeneous case from the non homogeneous.

If $I$ is a homogenous ideal, the variety $V$ has a conic structure (it is formed by a union of lines that go through the origin). It means that $V$ is contractible and hence its Euler characteristic is $1$.

For the non homogeneous case, consider the projection
\[
\begin{array}{rcl}
 \pi:\bc^n & \to & \bc^d \\
(x_1,\ldots,x_n) & \mapsto & (x_1,\ldots x_d)\end{array}
\]

We may assume (aplying a generic linear change of coordinates if necessary) that the following condition is satisfied:
\begin{dfn}
Consider $I_h\trianglelefteq\bc[x_0,x_1,\ldots,x_n]$ the homogenization of $I$. We will say that $I$ is in \textbf{general position} if $\sqrt{I_h+(x_0,x_1,\ldots,x_d)}\supseteq(x_0,x_1,\ldots,x_n)$.
\end{dfn}
\begin{thm}
The previous condition is satisfied by any ideal $I$ after a generic linear change of variables. Moreover, when this condition is satisfied the map $\pi$ restricted to $V$ is surjective and has no vertical asymptotes.
\end{thm}
\begin{proof}
 If we consider the projectivization $\bar{V}\subseteq \bc\bp^n$, the projection $\pi$ consists on taking as a center the $n-d-1$ dimensional subspace $S=\{[x_0:x_1:\cdots:x_n]\mid x_0=x_1=\cdots=x_d=0\}$. Since the dimension of $\bar{V}$ is $d$, the intersection $\bar{V}\cap S$ is generically empty. We may hence assume that after a generic linear change of variables, $S\cap \bar{V}=\emptyset$. This intersection is given precisely by the ideal $\sqrt{I_h+(x_0,x_1,\ldots,x_d)}$, which is homogeneous. The condition of $\bar{V}\cap S$ being empty is equivalent to the ideal $I$ being in general position.

 In this situation the preimage by $\pi\mid_V$ of a point $[1:x_1:\cdots:x_d]$ is given by the intersection of the subspace $\{[y_0:y_1:\cdots :y_n] \mid y_1=y_0x_1,\ldots,y_d=y_0 x_d\}$ with $\bar{V}$. By the genericity assumption, this intersection does not have points in the infinity. By dimension arguments, this intersection cannot be empty, and must be contained in the affine part of $\bar{V}$. We have then proved that $\pi$ restricted to $V$ is surjective.

\end{proof}

The intersection of a generic linear subspace of dimension $n-d$ with $\bar{V}$ is a union of $g$ distinct points, being $g$ the degree of $I_h$. This degree can be computed through the Hilbert polynomial. Since $I$ is in general position, all of the intersections of $\bar{V}$ with the fibres of $\pi$ will happen in the affine part. This means that $\pi$ restricted to $V$ is a branched cover of degree $g$. We will see now that the branching locus of this cover is contained in a subvariety of $\bc^d$ that can be computed.

Assume $I=(f_1,\ldots,f_s)$ is in general position. Consider the matrix
\[
 M:=\left(\begin{matrix} \frac{\partial f_1}{x_{d+1}} & \cdots & \frac{\partial f_1}{x_{n}} \\
           \vdots & \ddots & \vdots \\
\frac{\partial f_s}{x_{d+1}} & \cdots & \frac{\partial f_s}{x_{n}}
          \end{matrix}\right)
\]
and the ideal $J$ generated by its $(n-d)\times(n-d)$ minors.

\begin{thm}
The branching locus of $\pi\mid_V$ is contained in the elimination ideal $(I+J)\cap \bc[x_1,\ldots,x_d]$.

\end{thm}
\begin{proof}
Consider a point $p=(x_1,\ldots,x_n)\in V$. If the linear space $\pi^{-1}(\pi(p))$ intersects $V$ at $p$ transversally, then there is no ramification at $p$, since it means that at a neighbourhood of $p$ the map $\pi\mid_V$ is a difeomorphism. This condition of transversality can be expressed as follows: the normal space of $V$ in $p$ and the normal space of $\pi^{-1}(\pi(p))$ generate the tangent space of $\bc^n$ in $p$.
The normal space of $V$ in $p$ is generated by the rows of the matrix
\[ 
 \left( \begin{matrix}
         \frac{\partial f_1}{x_1}(p) & \cdots & \frac{\partial f_1}{x_n}(p)\\
\vdots & \ddots & \vdots\\
 \frac{\partial f_s}{x_1}(p) & \cdots & \frac{\partial f_s}{x_n}(p)\\

        \end{matrix}
\right).
\]
The normal space of $\pi^{-1}(\pi(p))$ is generated by the first $d$ vectors of the canonical basis. A gaussian elimination argument tells us that these two spaces generate the whole space if and only if the matrix $M$ has rank $(n-d)$. So the set of points of $V$ where $\pi\mid_V$ ramifies is contained in the set $S$ of zeros of $I+J$.

The elimination ideal $\bc[x_1,\ldots x_d]\cap (I+J)$ is the Zariski closure of $\pi(S)$. 

\end{proof}

Since $V\setminus \pi^{-1}(\pi(S))$ is a cover over $\bc^d\setminus \pi(S)$ of degree $g$, we have that
\[
 \chi(V)=\chi(V\setminus \pi^{-1}(\pi(S)))+\chi(V\cap \pi^{-1}(\pi(S)))=g\cdot \chi(\bc^d\setminus \pi(S))+\chi(V\cap \pi^{-1}(\pi(S)))=\]
\[
=g\cdot(\chi(\bc^d)-\chi(\pi(S)))+\chi(V\cap \pi^{-1}(\pi(S)))=g-g\cdot \chi(\pi(S))+\chi(V\cap \pi^{-1}(\pi(S))).
\]

Both $\pi(S)$ and $V\cap \pi^{-1}(\pi(S))$ are varieties of dimension smaller than $V$, so, by induction hypothesis, we can compute their Euler characteristic in the same form.

\section{Description of the algorithm}
\label{sec-descripcion}
Now we will describe, step by step, an algorithm to compute the Euler characteristic of the zero set of an ideal $I=(f_1,\ldots,f_s)$.

\begin{alg}
\label{algoritmoEuler}
(Compute the Euler characteristic of the algebraic set defined by the ideal $I$):

\begin{enumerate}
\item Check if $I$ is homogeneous. If it is, return $1$.

 \item Compute the associated primes $(I_1,\ldots,I_m)$ of $I$. This can be achieved by primary decomposition (see \cite[Chapter 4]{singular-book}).
\item If there is more than one associated prime, we have that
\[
 \chi(V(I))=\chi(V(I_1))+\chi(V(I_2\cap\cdots\cap I_m))-\chi(V(I_1+(I_2\cap\cdots \cap I_m))).
\]
by recursion, each summand can be computed with this algorithm. The following parts of this algorithm consider only the irreducible case, since we have already computed the associated primes, we will assume that $I_1$ is prime.
\item Compute the dimension $d$ and the degree $g$ of $V(I)$. If $d$ is zero, return $g$.

\item Check that $I$ is in general position. This can be done by computing a Gröbner basis of $\sqrt{I_h+(x_0,x_1,\ldots,x_d)}$ (where $I_h$ is the homogenization of $I$) and using it to check that $x_{d+1},\ldots,x_n$ are in it. If it is not in general position, apply a generic linear change of variables and start again the algorithm.

\item Construct the ideal $J$ generated by the $(n-d)\times (n-d)$ minors of the matrix
\[
 M:=\left(\begin{matrix} \frac{\partial f_1}{x_{d+1}} & \cdots & \frac{\partial f_1}{x_{n}} \\
           \vdots & \ddots & \vdots \\
\frac{\partial f_s}{x_{d+1}} & \cdots & \frac{\partial f_s}{x_{n}}
          \end{matrix}\right).
\]
\item Compute the elimination ideal $K=(I+J)\cap\bc[x_1,\cdots,x_d]$.
\item Compute by recursion $\chi(V(K))$ and $\chi(V(I+K))$. Return the number $g-g\cdot \chi(V(K))+\chi(V(I+K))$. 
\end{enumerate}
\end{alg}

\section{A finer invariant}
\label{sec-invariante}
The previous method essentially consists in decomposing our variety $V$ in pieces, each of which is compared to $\bc^i$ through linear maps that are unbranched covers. At the end of the day, it gives us a linear combination (with integer coefficients), of the Euler characteristic of $\bc^i$.

Now we will show that we can actually keep the information in this linear combination, defining a slightly different invariant. This information will be kept in a polynomial $F_\pi(V)\in\bz[L]$, where $L^i$ will play the role of $\chi(\bc^i)$

We follow the same method as before but with two differences:
\begin{itemize}
 \item If the ideal $I$ is homogeneous, we don't end returning a $1$. Instead, we continue the algorithm, taking as $I_h$ the ideal generated by $I$ inside $K[x_0,\ldots,x_n]$.
\item In the final step, we return $g\cdot L^d -g\cdot F_\pi(V(K))+F_\pi(V(I+K)))$ instead of $g-g\cdot \chi(V(K))+\chi(V(I+K))$.
\end{itemize}

So the algorithm results like this:

\begin{alg}
\label{algoritmopoly}
(Compute the polynomial $F\pi(V)$ associated to an algebraic set $V(I)$ in general position).

\begin{enumerate}

\item Compute the associated primes $(I_1,\ldots,I_m)$ of $I$.
\item If there is more than one associated prime, we have that
\[
 F_\pi(V(I))=F_\pi(V(I_1))+F_\pi(V(I_2\cap\cdots\cap I_m))-F_\pi(V(I_1+(I_2\cap\cdots \cap I_m))).
\]
by recursion, each summand can be computed with this algorithm. The following parts of this algorithm consider only the irreducible case, since we have already computed the associated primes, we will assume that $I_1$ is prime.
\item Compute the dimension $d$ and the degree $g$ of $V(I)$. If $d$ is zero, return $g$.

% \item Check that $I$ is in general position. This can be done by computing a Gröbner basis of $\sqrt{I_h+(x_0,x_1,\ldots,x_d)}$ (where $I_h$ equals to $\bc[x_0,\ldots,x_1]\cdot I$ if $I$ is homogeneous, and to the homogenization of $I$ otherwise) and using it to check that $x_{d+1},\ldots,x_n$ are in it. If it is not in general position, return error.

\item Construct the ideal $J$ generated by the $(n-d)\times (n-d)$ minors of the matrix
\[
 M:=\left(\begin{matrix} \frac{\partial f_1}{x_{d+1}} & \cdots & \frac{\partial f_1}{x_{n}} \\
           \vdots & \ddots & \vdots \\
\frac{\partial f_s}{x_{d+1}} & \cdots & \frac{\partial f_s}{x_{n}}
          \end{matrix}\right).
\]
\item Compute the elimination ideal $K=(I+J)\cap\bc[x_1,\cdots,x_d]$.
\item Compute by recursion $F_\pi(V(K))$ and $F_\pi(V(I+K))$. Return the number $gL^d-g\cdot F_\pi(V(K))+F_\pi(V(I+K))$. 
\end{enumerate}
\end{alg}

Note that both algorithms~\ref{algoritmoEuler} and \ref{algoritmopoly} can run differently if we apply a linear change of coordinates to $I$ (which would change the projection $\pi$). The topological properties of the Euler characteristic tells us that the final result of the algorithm~\ref{algoritmoEuler} will coincide with the Euler characteristic regardless of this linear change of coordinates. But in the case of $F_\pi(V)$ we cannot ensure such a result. Nevertheless, for two sufficiently generic projections, algorithm~\ref{algoritmopoly} will follow the same exact steps, so we can define $F(V)$ as the polynomial obtained by the algorithm~\ref{algoritmopoly} for generic projections. 

More preciselly, there must exist a Zariski open set $T\subseteq GL(n,\bc)$ such that, the polynomial $F_\pi(\sigma(V(I)))$ is the same for every linear change of coordinates $\sigma\in T$. 

\begin{dfn}
Given an ideal $I\trianglelefteq \bc[x_1,\ldots,x_n]$, we define the polynomial $F(V(I))$ as the polynomial $F_\pi(\sigma(V(I))$ for any $\sigma\in T$.

We will say that $I$ or $V(I)$ are in \textbf{generic position}, or that we are in \textbf{generic coordinates} if $F_\pi(V(I))=F(V(I))$.
\end{dfn}

Since so far we have no algorithmic criterion to determine if a projection is generic enough or not, the generic case can be computed by introducing the parameters of the projection, and computing the Gröbner basis with those parameters.

Anyways, experimental evidence suggests the following conjecture:
\begin{conj}\label{conjetura}
If an ideal is in general position, it is also in generic position.
\end{conj}

Some partial results in this direction are easy to show.
\begin{lema}
 If $V$ in in general position, the leading term of $F_\pi(V)$ coincides with the leading term of $F(V)$.
\end{lema}
\begin{proof}
 It is immediate to check that the degree of $F_\pi(V)$ coincides with the dimension of $V$, and that the leading coefficient of $F_\pi(V)$ coincide with the degree of $V$, regardless of the projection used to compute it.
\end{proof}
\begin{obs}
 The value of $F_\pi(V)$ at $L=1$ equals $\chi(V)$, independently of the choices of projections made for its computation, as long as we are in general position.
\end{obs}

These two results actually shows that $F(V)$ is independent of the projection for the case of curves (since in this case it is a degree $1$ polynomial whose leading term and value at $1$ are fixed).

We will now show that the invariant $F_\pi$ behaves well with respect to the product of varieties:
\begin{prop}
 Let $I_1\trianglelefteq \bc[x_1,\ldots,x_n]$ and $I_2\trianglelefteq \bc[y_1,\ldots,y_m]$ be two ideals on polynomial rings with separated variables, and let $V_1\subseteq \bc^n$ and $V_2\subseteq \bc^m$ be their corresponding algebraic sets of dimensions $d_1$ and $d_2$ respectively. Consider the ideal  \[I:=I_1+I_2\trianglelefteq\bc[x_1\ldots,x_{d_1},y_1,\ldots,y_{d_2},x_{d_1+1},\ldots,x_n,y_{d_2+1},\ldots,y_m].\] Its corresponding algebraic set is $V=V_1\times V_2\subseteq\bc^n\times \bc^m=\bc^{n+m}$. Then $F_\pi(V)=F_\pi(V_1)\cdot F_\pi(V_2)$.
\end{prop}
\begin{proof}
Without loss of generality, we can assume that we are in the irreducible case. We will work on induction over the dimension. If $V_1$ or $V_2$ ar zero dimensional, the result is immediate.

Consider $g_1,g_2$ the degrees of $V_1$ and $V_2$, and $g$ the degree of $V$. It is easy to check that $g=g_1\cdot g_2$. Is is also immediate to check that, if $V_2=\bc^m$, the statements holds (that is, $F_\pi(V_1\times \bc^m)=F_\pi(V_1)\cdot L^m$). Consider also $\Delta_1,\Delta_2$ and $\Delta$ analogously. 

Now we will show that $\Delta=(\Delta_1\times\bc^{d_2})\cup (\bc^{d_1}\times \Delta_2)$. Let $p=(x_1,\ldots,x_{d_1},y_1,\ldots,y_{d_2})\in\bc^{d_1}\times \bc^{d_2})$. The set of points in $V$ that project on $p$ is the product of the set of points in $V_1$ that project in $(x_1\,\ldots,x_{d_1})$ and the set of points in $V_2$ that project in $(y_1,\ldots,y_{d_2})$. This set has less than $g_1\cdot g_2$ points if and only if $(x_1,\ldots,x_{d_1})\in \Delta_1$ or $(y_1,\ldots,y_{d_1})\in \Delta_2$. It is immedaite also that $(\Delta_1\times\bc^{d_2})\cap (\bc^{d_1}\times \Delta_2)=\Delta_1\times \Delta_2$. By induction hypothesis, we have that
\[
 F_\pi(\Delta)=L^{d_2}\cdot F_\pi(\Delta_1)+L^{d_1}\cdot F_\pi(\Delta_2)-F_\pi(\Delta_1)\cdot F_\pi(\Delta_2).
\]

Reasoning analogoulsy, we can conclude that
\[
 F_\pi(\pi^{-1}(\Delta))=F_\pi(\pi^{-1}(\Delta_1))\cdot F_\pi(V_2)+F_\pi(\pi^{-1}(\Delta_2))\cdot F_\pi(V_1)-F_\pi(\pi^{-1}(\Delta_1))\cdot F_\pi(\pi^{-1}(\Delta_2)).
\]

So, sumarizing, we have that
\[
\begin{array}{rcl} F_\pi(V)&=&g_1 g_2(L^{d_1+d_2}-F_\pi(\Delta))+F_\pi(\pi^{-1}(\Delta))\\
 F_\pi(V_1)&=&g_1(L^{d_1}-F_\pi(\Delta_1))+F_\pi(\pi^{-1}(\Delta_1)) \\
 F_\pi(V_2)&=&g_2(L^{d_2}-F_\pi(\Delta_2))+F_\pi(\pi^{-1}(\Delta_2)).
\end{array}
\]
Using all the previous formulas one can easily check that $F_\pi(V)=F_\pi(V_1)\cdot F_\pi(V_2)$.
\end{proof}

If conjecture~\ref{conjetura} is true, the same result would be true for $F(V)$. In fact, a weaker condition would be enough: if the product of two generc projections is generic, then the invariant $F$ is multiplicative. This could be useful, for example, to give a criterion to check if an algebraic set can be the product of two nontrivial algebraic sets. If $F(V)$ is irreducible in $\bz[L]$, then $V$ couln't be a product.

\section{The projective case}
\label{sec-proyectivo}
To compute the Euler characteristic of the projective variety $\bar{V}$ defined by a homogeneous ideal $I_h\trianglelefteq\bc[x_0,\ldots,x_n]$ we can also use algorithm~\ref{algoritmoEuler}. In order to do so, we will consider the hyperplane $H$ ``at infinity'' given by the equation $x_0=0$. This allows us to decompose $\bar{V}$ as its affine part $V:=\bar{V}\setminus H$ and its part at infinity $\bar{V}^\infty :=\bar{V}\cap H$. It is clear that $\chi(\bar{V})=\chi(V)+\chi(\bar{V}^\infty)$.

The affine part $V$ is an affine variety defined by the ideal obtained by substituting $x_0=1$ in the generators of $I_h$, whose Euler characteristic can be defined as seen before.

The part at infinity $\bar{V}^\infty$ is a projective variety embedded in a projective space of less dimension. The homogeneous ideal that defines is obtained by substituting $X_0=0$ in the generators of $I_h$. Its Euler characteristic can be computed by recursion. If we are in the case of $\bc\bp^1$, $\bar{V}$ will consist on a finite number of points, which can be computed as the degree of $\sqrt{I_h}$.

Now we will show a different way to compute the Euler characteristic of a projective variety using the polynomial $F(V)$.

\begin{thm}
 Let $I=\bc[x_0,\ldots,x_n]\cdot(f_1,\ldots,f_s)$ be a homogeneous ideal in generic position. Assume that the generators $f_1,\ldots,f_s$ are homogeneous. Denote by $I_0:=(I+\bc[x_0,\ldots,x_n]\cdot(x_0))\cap \bc[x_1,\ldots,x_n]$, and $I_1:=(I+\bc[x_0,\ldots,x_n]\cdot(x_0-1))\cap \bc[x_1,\ldots,x_n]$. That is, the ideals that represent the intersection of $V(I)$ with the hyperplanes $\{x_0=0\}$ and $\{x_0=1\}$ respectively, seeing the two hyperplanes as ambient spaces. Then the following formula holds:

\[
 F(V(I))=(L-1)\cdot F(V(I_1))+F(V(I_0))
\]

\end{thm}
\begin{proof}
 By induction on the dimension of $V(I)$.
If the dimension is zero, $V(I)$ must consist only on the origin, since $I$ is homogeneous. In this case, $V(I_1)$ is empty, and $V(I_0)$ is also the origin. We have that
\[
 1=F(V(I))=(L-1)\cdot 0+1=(L-1)\cdot F(V(I_1))+F(V(I_0)).
\]

If the dimension $d$ of $V(I)$ is positive, consider the ideals $J,K$ and $H=I+K$ as before. Construct also $H_0,H_1,K_0$ and $K_1$ in the same way as $I_0$ and $I_1$. Note that, since we are in generic position, the ideals $H_0'$ and $K_0'$ needed to compute $F(V(I_0))$ are precisely $H_0$ and $K_0$ (that is, specializing $x_0=0$ and then computing the minors of the matrix $M$ and the elimination ideal is the same as computing the minors of the matrix and the elimination and then specializing). The same happens with $J_1$ and $K_1$.

By induction hypothesis, we have that
\[
F(V(H))=(L-1)\cdot F(V(H_1))+F(V(H_0))\]
and
\[
F(V(K))=(L-1)\cdot F(V(K_1))+F(V(K_0)).
\]

Now we have that
\[
\begin{array}{rcl}
F(V(I)) & = & g\cdot L^d-g\cdot F(V(K))+F(V(H))= \\
 & = & g\cdot L^d-g\cdot((L-1)\cdot F(V(K_1))+F(V(K_0)))+(L-1)\cdot F(V(H_1))+F(V(H_0))=\\
 & = & (L-1)\cdot (g\cdot L^{d-1}-g\cdot F(V(K_1))+F(V(H_1))+g\cdot L^{d-1}-g\cdot F(K_0)+F(H_0)=\\
& = & (L-1)\cdot F(V(I_1))+F(V(I_0))
\end{array}
\]
and this proves the result.
\end{proof}

This theorem allows to relate the invariant $F$ of the affine algebraic set defined by a homogenous ideal, and the invariant $F$ of the projective variety defined by the same ideal as follows:

\begin{cor}
Let $I$ be a homogeneous ideal in $\bc[x_1,\ldots,x_n]$ in generic position. Let $V$ be the algebraic set defined by $I$, and $V'$ the projective variety defined by the same ideal. Then $F(V)=(L-1)\cdot F(V')+1$.
\end{cor}
\begin{proof}
 Aplying the previous result recursively, we have that $F(V(I))$ equals
\[
(L-1)\cdot F(V(I_1))+(L-1)\cdot F(V((I_0)_1))+\cdots+(L-1)\cdot F(V((\cdots(I_0)\cdots)_0)_1))+F(V((\cdots (I_0)\cdots)_0)).
\]
Note that if we compute piecewise $F(V')$ we obtain preciselly $F(V(I_1))+F(V((I_0)_1))+\cdots+F(V((\cdots(I_0)\cdots)_0)_1))$. Since $V((\cdots (I_0)\cdots)_0))$ consists only on the origin, we have the result. 

This last result can be interpreted as the fact that $V$ is the complex cone over $V'$. That is, $V\setminus\{0\}$ is the product $V'\times \bc^*$.

\end{proof}

\section{Examples}
\label{sec-ejemplos}
Both the polynomial $F(V)$ and the Hilbert polynomial $P_I$ have the same degree, and the leading term determined by the degree of $V$. This would point in the direction of considering that they contain the same information. The following example shows this is not the case:

\begin{ejm}
 Consider the conics $C_1,C_2\in\bc^2$ given by $C_1:=V(x^2+y^2-1)$ and $C_2:=V(x^2+y^2)$. If we compute the Hilbert polynomial of the corresponding homogenous ideals in $\bc[x,y,z]$ we get that $P_{(x^2+y^2+z^2)}=P_{(x^2+y^2)}=2\cdot t+1$.

Let's now compute the polynomial $V(C_1)$ using the canonical projection $\pi:\bc^2\to\bc$ to the first component. This projection is a $2:1$ cover of $\bc$ branched along the points $\pm1$. Over each of these points, there is only one preimage. So finally we have that $F(C_1)=2(L-2)+1+1=2L-2$.

On the other hand, the curve $C_2$ also projects $2:1$ over $\bc$, but now there is only one branching point ($x=0$). So the result is $F(C_2)=2(L-1)+1=2L-1$.

That is, this example shows that the polynomial $F(V)$ contains information that is not contained in the Hilbert polynomial.
\end{ejm}

An important example of algebraic sets is the case of hyperplane arrangements. We will now recall some related notions now (see~\cite[Chapter II]{orlik-terao}).

\begin{dfn}
Let $\mathcal A$ be a hyperplane arrangement in $\bc^n$. Its \textbf{intersection lattice} $L(\mathcal A)$ is the set of all its intersections ordered by reverse inclusion, with the convention that the intersection of the empty set is $\bc^n$ itself.

The Möbius function is the only function $\mu:L\rightarrow \bz$ satisfying that:
\[\begin{array}{cc}
 \mu(\bc^n)=1 & \\
\sum_{Y\leq X}\mu(Y)=0 & \forall X\in L\setminus\{\bc^n\}.
\end{array}
\]
\end{dfn}
\begin{dfn}
 The \textbf{characteristic polynomial} of $\mathcal A$ is defined as
\[
 \chi(\mathcal{A},L):=\sum_{X\in L(\mathcal{A})}\mu(X)\cdot L^{dim(X)}.
\]
\end{dfn}

\begin{thm}[Deletion-Restriction]
 Let $\mathcal A$ be a hyperplane arrangement in $\bc^n$, and $H$ a hyperplane of $\mathcal{A}$. Let $\mathcal{A}'$ be the arrangement resulting from eliminating $H$ from $\mathcal{A}$, and let $\mathcal{A}''$ be the hyperlane arrangement inside $H$ induced by the intersection with $\mathcal{A}'$. Then the following formula holds:
\[
 \chi(\mathcal{A},L)=\chi(\mathcal{A}',L)-\chi(\mathcal{A}'',L).
\]
\end{thm}

Now we will see how this characteristic polynomial relates to the polynomial $F(V)$.

\begin{thm}
Let $\mathcal A$ be a hyperplane arrangement in $\bc^n$. The following holds:
\[
 F(\mathcal{A})=L^n-\chi(\mathcal{A},L).
\]
\end{thm}
\begin{proof}
 By induction on the number of hyperplanes. he case of only one hyperplane is immediate.

If there are more than one hyperplane, take one hyperplane $H$ of $\mathcal{A}$. Since $\mathcal{A}=\mathcal{A}'\cup H$, we have, by aditivity, that
\[
 F(\mathcal{A})=F(\mathcal{A}')+F(H)-F(\mathcal{A}'\cap H)=F(\mathcal{A}')+L^{n-1}-F(\mathcal{A}'').
\]
Both $\mathcal{A}'$ and $\mathcal{A}''$ are hyperplane arrangements with less hyperplanes than $\mathcal{A}$, so, by induction hypothesis the following formulas hold:
\[
 \begin{array}{c}
F(\mathcal{A}')=L^n-\chi(\mathcal{A}',L) \\
F(\mathcal{A}'')=L^{n-1}-\chi(\mathcal{A}'',L).
 \end{array}
\]
Substituting these formulas in the previous one, and using the delition-restriction theorem we get the result.
\end{proof}

This esult tells us that the characteristic polynomial of $\mathcal A$ can be though of as the polynomial $F$ of its complement.

\section{Code and timings}
\label{sec-codigo}
Here we show an implementation in Sage (\cite{sage}) of the two algorithms.
 
\subsection{Implementation of Algorithm~\ref{algoritmoEuler}}

\begin{verbatim}
def Euler_characteristic(I):
    R=I.ring()
    if I.is_one():
        return 0
    if R.ngens()==1:
        return sum([j[0].degree() for j in I.gen().factor()])
    J1=I.radical()
    if J1.is_homogeneous():
        return 1
    primdec=J1.associated_primes()
    J1=primdec[0]
    m=len(primdec)
    if m>1:
        J2=R.ideal(1)
        for j in [1..m-1]:
            J2=J2.intersection(primdec[j])
        return Euler_characteristic(J1)+Euler_characteristic(J2)-Euler_characteristic(J1+J2)
    P=J1.homogenize().hilbert_polynomial()
    if P.is_zero():
        deg=0
    else:
        deg=P.leading_coefficient()*P.degree().factorial()
    if deg==1:
        return 1
    dim=J1.dimension()
    n=R.ngens()
    vars1=R.gens()[0:n-dim]
    vars2=R.gens()[n-dim:n]
    varpiv=vars1[-1]
    IH=J1.homogenize()
    S=IH.ring()
    JH=IH+S.ideal(S.gens()[n-dim:])
    if JH.dimension()>0:
        det=0
        while det==0:
            MH=random_matrix(R.base_ring(),n)
            det=MH.determinant()
        L=list(MH*vector(list(R.gens())))
        return Euler_characteristic(R.hom(L)(J1))
    if dim==0:
        return deg
    M=matrix([[f.derivative(v) for v in vars1] for f in J1.gens()])
    J=R.ideal(M.minors(n-dim))
    K=(J+J1).elimination_ideal(vars1)
    S=PolynomialRing(R.base_ring(),vars2)
    H=R.hom([S(0) for j in vars1]+[S(j) for j in vars2])
    C=deg-deg*Euler_characteristic(H(K))+Euler_characteristic(K+J1)
    return C
\end{verbatim}
This algorithm may be very slow (since it involves several Gröbner basis computations), but in several interesting cases, it gives a useful answer in reasonable time. Let's show here some examples.

In the case of curves and surfaces, the result is often reasonably fast; but it may vary a lot if a random change of variables has to be aplied. Here we show a few examples. These tests have been run on a Dual-Core AMD Opteron 8220.

Three examples of plane curves:
\begin{verbatim}
sage: R.<x,y>=QQ[]
sage: time Euler_characteristic(R.ideal(x^5+1))
5
Time: CPU 0.16 s, Wall: 0.17 s
sage: time Euler_characteristic(R.ideal(y^4+x^3-1))
-5
Time: CPU 0.19 s, Wall: 0.20 s
time Euler_characteristic(R.ideal(x^2+y^2-5*x^2*y^4+x*y-1))
-8
Time: CPU 17.82 s, Wall: 17.82 s
\end{verbatim}
A curve and a surface in $\bc^3$:
\begin{verbatim}
S.<x,y,z>=QQ[]
timeit('Euler_characteristic(S.ideal(x^5+y^2+2*x*y+1,3*x-5*y*x+y^2+1))')
10
Time: CPU 0.17 s, Wall: 0.18 s
timeit('Euler_characteristic(S.ideal(x^5+y^2+2*x*y+1))')
-3
Time: CPU 0.49 s, Wall: 0.49 s
\end{verbatim}

\subsection{Implementation of Algorithm~\ref{algoritmopoly}}

\begin{verbatim}
@parallel(7)
@cached_function
def FV(I,var='L'):
    FS=PolynomialRing(ZZ,var)
    L=FS.gen()
    R=I.ring()
    if R.ngens()==0:
        return 0
    if I.is_zero():
        return L^R.ngens()
    if I.is_one():
        return 0
    if R.ngens()==1:
        return FS(sum([j[0].degree() for j in I.gen().factor()]))
    J1=I.radical()
    if J1.is_homogeneous():
        S1=PolynomialRing(R.base_ring(),R.gens()[0:-1])
        H1=R.hom(list(S1.gens())+[S1(1)])
        H2=R.hom(list(S1.gens())+[S1(0)])
        resulp=FV([H1(J1),H2(J1)])
        d=dict([[a[0][0][0],a[1]] for a in resulp])
        [i1,i2]=[d[H1(J1)],d[H2(J1)]]
        return (L-1)*(i1)+i2
    primdec=J1.associated_primes()
    J1=primdec[0]
    m=len(primdec)
    if m>1:
        J2=R.ideal(1)
        for j in [1..m-1]:
            J2=J2.intersection(primdec[j])
            resulp=FV([J1,J2,J1+J2])
            d=dict([[a[0][0][0],a[1]] for a in resulp])
            [i1,i2,i3]=[d[J1],d[J2],d[J1+J2]]
        return i1+i2-i3
    P=J1.homogenize().hilbert_polynomial()
    if P.is_zero():
        deg=0
    else:
        deg=P.leading_coefficient()*P.degree().factorial()
    dim=J1.dimension()
    if deg==1:
        return FS(L^dim)
    n=R.ngens()
    vars1=R.gens()[0:n-dim]
    vars2=R.gens()[n-dim:n]
    varpiv=vars1[-1]
    IH=J1.homogenize()
    S=IH.ring()
    JH=IH+S.ideal(S.gens()[n-dim:])
    if JH.dimension()>0:
        det=0
        while det==0:
            MH=random_matrix(R.base_ring(),n)
            det=MH.determinant()
        L=list(MH*vector(list(R.gens())))
        return FV(R.hom(L)(J1))
    if dim==0:
        return FS(deg)
    M=matrix([[f.derivative(v) for v in vars1] for f in J1.gens()])
    J=R.ideal(M.minors(n-dim))
    K=(J+J1).elimination_ideal(vars1)
    S=PolynomialRing(R.base_ring(),vars2)
    H=R.hom([S(0) for j in vars1]+[S(j) for j in vars2])
    d=dict([[a[0][0][0],a[1]] for a in FV([H(K),K+J1])])
    [i1,i2]=[d[H(K)],d[K+J1]]
    C=deg*FS(L^dim-i1)+i2
    return C
\end{verbatim}
This implementation makes use of the Sage framework for parallel computations, allowing to use several processor cores at the same time to compute the intermediate steps. It also caches the already computed results in case they would be needed later. 

Note that, if we assume Conjecture~\ref{conjetura} to be true, this implementation works fine giving as entry the ideal whose algebraic set we want to compute (assuming it is in general position, which is something that can be easily checked). If we want to be safe from the posibility of the conjecture to be false, we have to introduce it over a ring that contains the parameters of the possible linear transformations. But, unluckily, the methods to compute the primary decomposition do not work over rings with parameters. One way to proceed is to compute the primary decomposition over the original ring without parameters. Then compute the discriminant with parameters, and then choose a value for the parameters in the open part of the Gröbner cover (see~\cite{montes-groebnercover} for a definition and algorithm).

Again, this method can be very slow, but in some cases it is fast enough to be useful. Here we present some of those examples:

The complex $2$-sphere:
\begin{verbatim}
sage: R.<x,y,z>=QQ[]
sage: time FV(R.ideal(x^2+y^2+z^2-1))
2*L^2 - 2*L + 2
Time: CPU 0.07 s, Wall: 0.32 s
\end{verbatim}

Another surface:
\begin{verbatim}
sage: time FV(R.ideal(x^3+y^3+z^3-1))
3*L^2 - 6*L + 12
Time: CPU 0.07 s, Wall: 0.29 s
\end{verbatim}

The intersection of the two:
\begin{verbatim}
sage: time FV(R.ideal(x^3+y^3+z^3-1,x^2+y^2+z^2-1))
6*L - 15
Time: CPU 0.08 s, Wall: 43.01 s
\end{verbatim}
and their union (the timing is done after cleaning the cache of the function):
\begin{verbatim}
sage: time FV(R.ideal((x^3+y^3+z^3-1)*(x^2+y^2+z^2-1)))
5*L^2 - 14*L + 29
Time: CPU 0.08 s, Wall: 43.02 s
\end{verbatim}

notice the aditivity of the polynomial.

The Whitney umbrella:
\begin{verbatim}
sage: time FV(R.ideal(x*y^2-z^2))
3*L^2 - 4*L + 2
Time: CPU 0.13 s, Wall: 3.44 s
\end{verbatim}

The $3$-sphere:
\begin{verbatim}
sage: S.<x,y,z,t>=QQ[]
sage: time FV(S.ideal(x^2+y^2+z^2+t^2-1))
2*L^3 - 2*L^2 + 2*L - 2
Time: CPU 0.09 s, Wall: 0.43 s
\end{verbatim}

\section{Counting points over finite fields}
\label{sec-cuerposfinitos}
In this section we will see how the polynomial $P(V)$ can be related to the number of points of the variety considered over a finite field. Let's ilustrate this fact with an example.
\begin{ejm}
Consider the conic given by the equation
\[
 x_1^2+x_2^2-1
\]
in the affine plane over the field of $5$ elements $\bfff_5$.

The set of rational points is the following:
\[
 (0,1),(0,4),(1,0),(4,0)
\]

If we project it to the $x_1$ axis, we see that over the point $(0)$, we have two preimages, as expected by the degree. Over the points $(1)$ and $(4)$, we have just one point of the curve, since the cover ramifies there. But over the points $(2)$ and $(3)$ we have no points of the curve. The reason for this is that the equations $2^2+x_2^2-1$ and $3^2+x_2^2-1$ have no roots over $\bfff_5$. However, they do have all their solutions over a quadratic extension $\bfff_{25}$ of degree $2$. In particular, if we look at the points of $\bfff_5\times\bfff_{25}$ that satisfy the equation, we obtain the set
\[
 \{(0, 1), (0, 4), (1, 0), (2, a + 2), (2, 4a + 3), (3, a + 2), (3, 4a + 3), (4, 0)\},
\]
where $a$ is an element of $\bfff_{25}$ that has minimal polynomial $x^2+3$ over $\bfff_5$.

It might seem strange to consider the set of points in $\bfff_{5}\times \bfff_{25}$ that satisfy a given equation. But this set is in fact an algebraic set. Indeed, it can be expressed as the set of points of the affine plane over $\bfff_{25}$ that satisfy the equations
\[\begin{array}{c}
 x_1^2+x_2^2-1\\
x_1^5-x_1
\end{array}
\]

Recall that the polynomial $P(V)$ for such a conic was $f=2L-2$. In this case, we have obtained $8$ points in total, which is preciselly the value of $f$ for $L=5$. A quick look at the points shows why this happens: there are two points over each value of $\bfff_5$, with the exception of the two branching points, where there is only one.
\end{ejm} 

Note that both algorithms~\ref{algoritmoEuler} and~\ref{algoritmopoly} can be run over finite fields in the same way as they run over the rationals. Whith a small exception, though: to ensure the existence of a change of coordinates that puts the ideal in general position, we might need to work in a finite field extension. Once done that, both algorithms would mimic the steps given by the algorithm run over $\bq$, except if some leading coefficient becomes zero. But that would happen only for a finite number of prime numbers $p$.
\begin{ntc}
We can then define the polynomial $F_p(V)$ as the result of runing algorithm \ref{algoritmopoly} in a field of characteristic $p$. 
\end{ntc}
As we have seen before, $F_p(V)=F(V)$ for almost every prime number $p$.
\begin{ntc}
Given a polynomial \[f=a_0+a_1L+a_2L^2+\cdots+a_nL^n\in\bz[L]\] and a list of numbers $(d_1,\ldots,d_s)$ with $s\geq n$, we will denote by $f(d_1,\ldots,d_s)$ the number 
\[
f(d_1,\ldots,d_s)=a_0+a_1d_1+a_2d_1d_2+\cdots+a_nd_1d_2\cdots d_n
\]
\end{ntc}

Another difference between the way the algorithms would run over $\bq$ and over finite fields lies in the primary decomposition. But then again, this will only happen for a finite number of primes.

\begin{thm}
Given an ideal $I\trianglelefteq \bfff_p[x_1,\ldots,x_n]$, and the corresponding $f:=F_p(V)$ of degree $s$, there exists a list of numbers $(d_1\leq\cdots\leq d_n)$ such that the number of points in
\[\left(\bfff_{p^{d_1}}\times\cdots\times\bfff_{p^{d_n}}\right)\cap V(I)
\]
equals the number $f(p^{d_1},\ldots,p^{d_n})$.

Moreover, for any such list $(d_1\leq \cdots \leq d_n)$ and any $d_i'$ multiple of $d_i$ there exists another list $(d_1,\leq\cdots\leq d_{i-1}\leq d_i'\leq d_{i+1}'\leq \cdots \leq d_n')$ satisfying the same property.
\end{thm}

\begin{proof}
 By induction over the dimension. If $I$ is zero dimensional, $V(I)$ consists on $F_p(V)=deg(I)$ distinct points, whose coordinates lie in a suficiently big extension $\bfff_{p^{d_1}}$. Any sequence starting with $d_1$ would satisfy the theorem.

If $I$ has dimension $d>0$ and degree $g$, we have that $F_p(V(I))=g(L^d-F_p(V(K)))+F_p(V(I+K))$. By induction, we can assume that both $K$ and $I+K$ satisfy the result. Let $(d_1^1\leq\cdots\leq d_d^1)$ and $(d_1^2\leq\cdots\leq d_n^2)$ the corresponding sequences. Take $d_1=gcd(d_1^1,d_1^2)$. There exist two sequences $(d_1\leq {d_2^1}'\leq\cdots\leq {d_d^1}')$ and $(d_1\leq {d_2^2}'\leq\cdots\leq {d_n^2}')$ that are valid for $K$ and $I+K$ respctivelly, and coinciding in the first term. Repeating this reasoning we can obtain two sequences $(d_1\leq\cdots \leq d_d)$ and $(d_1\leq\cdots\leq d_d\leq d_{d+1}\leq \cdots \leq d_n)$ that are valid for $K$ and $I+K$ respectivelly.

Note that, since both $F_p(V(K))$ and $F_p(V(I+K))$ are of dimension at most $d-1$, the terms $d_d,\ldots, d_n$ can be changed arbitrarily and still the squence would be valid for $K$ and $I+K$.

Now take any point $q:=(q_1,\ldots,q_d)\in(\bfff_{p^{d_1}}\times\cdots\times \bfff_{p^{d_{d}}})\setminus V(K)$. Taking an apropiate field $\bfff_{p^{n_q}}$, we can ensure that there are exactly $g$ points in $\{(x_1,\ldots,x_n)\in (\bfff_{p^{n_q}})^n\mid x_1=q_1,\ldots,x_d=q_d\}\cap V$. We can do the same for every point $q$, and take a common field extension $\bfff_{p^s}$ of all the different $\bfff_{p^{n_q}}$. This way, we have that there are exactly $g$ points of $V\cap (\bfff_{p^{d_1}}\times\cdots\times \bfff_{p^{d_{d}}}\times \bfff_{p^{s}}\times\cdots \times \bfff_{p^{s}})$ over each point of $(\bfff_{p^{d_1}}\times\cdots\times \bfff_{p^{d_{d}}})\setminus V(K)$.

By definition, we have that 
\[
 F_p(V)=g(L^d-F_p(V(K)))+F_p(V(I+K)).
\]

Now if we restrict ourselves to the points in $\bfff_{p^{d_1}}\times\cdots\times \bfff_{p^{d_{d}}}\times \bfff_{p^{s}}\times\cdots \times \bfff_{p^{s}}$, we have that
\[
 \#V=g\cdot(p^{d_1}p^{d_2}\cdots p^{d_d}-\#V(K))+\#V(I+K).
\]
Making use of the induction hypothesis, and the fact that $F_p(V(K))$ and $F_p(V(I+K))$ are of dimension less than $d$, the result follows easily.

\end{proof}

\begin{cor}
Let $V$ be an algebraic set in $\bc^n$ defined by a an ideal $I\trianglelefteq\bk[x_1,\ldots,x_n]$, where $\bk$ is an algebraic extension of $\bq$. Then for almost every prime $p$ there exists a list of positive integers $(d_{1,p},\ldots,d_{n,p})$ such that the number of points in
\[\left(\bfff_{p^{d_1}}\times\cdots\times\bfff_{p^{d_n}}\right)
\] that satisfy the equations of $I$
equals the number $F(V(I))(p^{d_1},\ldots,p^{d_n})$.
\end{cor}

This corollary could allow a different way to compute the polynomial $F(V(I))$ by counting points over finite fields. If we know the value of $F(V(I))(S)$ for a sufficient number of such sequences $S$, recovering the coefficients of the polynomial is a simple linear algebra problem.

\bibliographystyle{abbrv}
\bibliography{biblio}
\end{document}